\documentclass[a4, 11pt] {amsart} 
\usepackage{mathrsfs}
\usepackage{amsmath}
\usepackage{amssymb}
\usepackage[all]{xy}
\usepackage[dvips]{graphicx, color}
\usepackage{amsmath}
\usepackage{amssymb}
\usepackage[all]{xy}
\usepackage{geometry}
\usepackage[dvips]{graphicx, color}
\usepackage{color}
\usepackage{mathrsfs}  
\usepackage{enumitem}
\begin{document}

\numberwithin{equation}{section}
\newtheorem{thm}[equation]{Theorem}
\newtheorem{pro}[equation]{Proposition}
\newtheorem{prob}[equation]{Problem}
\newtheorem{qu}[equation]{Question}
\newtheorem{cor}[equation]{Corollary}
\newtheorem{con}[equation]{Conjecture}
\newtheorem{lem}[equation]{Lemma}
\theoremstyle{definition}
\newtheorem{ex}[equation]{Example}
\newtheorem{defn}[equation]{Definition}
\newtheorem{ob}[equation]{Observation}
\newtheorem{rem}[equation]{Remark}

\hyphenation{homeo-morphism} 

\newcommand{\calA}{\mathcal{A}} 
\newcommand{\calD}{\mathcal{D}} 
\newcommand{\calE}{\mathcal{E}}
\newcommand{\calC}{\mathcal{C}} 
\newcommand{\Set}{\mathcal{S}et\,} 
\newcommand{\Top}{\mathcal{T}\!op \,}
\newcommand{\Topst}{\mathcal{T}\!op\, ^*}
\newcommand{\calK}{\mathcal{K}} 
\newcommand{\calO}{\mathcal{O}} 
\newcommand{\calS}{\mathcal{S}} 
\newcommand{\calT}{\mathcal{T}} 
\newcommand{\Z}{{\mathbb Z}}
\newcommand{\C}{{\mathbb C}}
\newcommand{\CC}{{\mathbb C}}
\newcommand{\Q}{{\mathbb Q}}
\newcommand{\R}{{\mathbb R}}
\newcommand{\N}{{\mathbb N}}
\newcommand{\F}{{\mathcal F}}
\newcommand{\PP}{{\mathbb P}}
\newcommand{\RR}{{\mathbb R}}
\def\op{\operatorname}
\hfill

\title[]
{Self-products of rationally elliptic spaces \\
and inequalities between the ranks of \\
homotopy and homology groups}

\author{Anatoly Libgober and Shoji Yokura}

\thanks{
\\
}

\date{}

\address{Department of Mathematics, University of Illinois, Chicago, IL 60607}

\email{libgober@math.uic.edu}

\address{
Graduate School of Science and Engineering, Kagoshima University, 1-21-35 Korimoto, Kagoshima, 890-0065, Japan
}

\email{yokura@sci.kagoshima-u.ac.jp}

\maketitle 

%\tableofcontents

\begin{center}
\emph{ Dedicated to the memory of Stefan Papadima}
\end{center}

\begin{abstract}
We give a survey on recent results on inequalities between 
the ranks of homotopy and cohomology groups (resp., graded components of mixed Hodge structures on these groups)
of rationally elliptic spaces (resp., quasi-projective varieties which are
rationally elliptic). We also discuss a refinement of these results describing  
a new invariant of rationally elliptic spaces 
allowing to compare the ranks of homotopy and homology groups.
This invariant is a specialization of an invariant $r\left ( P(t),Q(t);\varepsilon \right )$ of a pair $\left (P(t),Q(t) \right )$ 
of polynomials with non-negative integer coefficients, describing 
the range of variable $r$ 
such that $rP(t)<Q(t)^r$ for all $t \ge \varepsilon$. This range is
related to the classical Lambert W-function $W(z)$. 
\end{abstract}
%%%%%%%%%%%%%%
\section{Introduction}

A rationally elliptic space is a simply connected 
topological space $X$ such that 

\noindent
$\dim \left(\pi_*(X) \otimes \Q \right) < \infty$ and $\dim  H_*(X;\Q) < \infty$, where we set
$$\pi_*(X)\otimes \mathbb Q:= \displaystyle \bigoplus_{i\ge 2} \pi_i(X) \otimes \mathbb Q, \quad  H_*(X;\mathbb Q):= \displaystyle \bigoplus_{i\ge 0} H_i(X;\mathbb Q).$$
This interesting class of spaces has received considerable
attention, but a complete picture of structure, geometry or
invariants of spaces in this class appears to be far from clear.
Very strong restrictions on the ranks of homotopy group
were found a long time ago by J. B. Friedlander and S. Halperin (cf. \cite{FH} and also
\cite{FHT} or \cite{FOT}). 

From now on we use the following notation:
\begin{itemize}
\item $\pi_{\mathrm{even}}(X)\otimes \mathbb Q:= \displaystyle \bigoplus_{k\ge 1} \pi_{2k}(X) \otimes \mathbb Q$, \quad $\pi_{\mathrm{odd}}(X)\otimes \mathbb Q:= \displaystyle \bigoplus_{k\ge 0} \pi_{2k+1}(X) \otimes \mathbb Q$,
\item $H_{\mathrm{even}}(X;\mathbb Q):= \displaystyle \bigoplus_{k\ge 0} H_{2k}(X;\mathbb Q)$, \quad $H_{\mathrm{odd}}(X;\mathbb Q):= \displaystyle \bigoplus_{k\ge 0} H_{2k+1}(X;\mathbb Q) $,
\item $H^{\mathrm{even}}(X;\mathbb Q):= \displaystyle \bigoplus_{k\ge 0} H^{2k}(X;\mathbb Q)$, \quad $H^{\mathrm{odd}}(X;\mathbb Q):= \displaystyle \bigoplus_{k\ge 0} H^{2k+1}(X;\mathbb Q) $.
\end{itemize}
To recall Friedlander--Halperin's results, let $x_i$ (resp. $y_j$) denote a basis of
$\pi_{\mathrm{odd}}(X)\otimes \Q$ (resp. $\pi_{\mathrm{even}}(X) \otimes \Q$) and let 
$n$
be the formal dimension of the space $X$, i.e., the maximal degree 
$n$ such that 
$H^n(X;\Q) \ne 0.$
Then we have the following:
%%%%%%%%%%%%%%%%%%
\begin{enumerate}[label=(\alph*)]
\item \label{a} $\displaystyle \sum_i\op{deg} x_i \le 2n-1, \sum_j \op{deg} y_j \le n$.
\item \label{b} $n= \displaystyle \sum_i  \op{deg} x_i-\sum_j (\op{deg} y_j-1)$.
\item \label{c} $\chi^{\pi}(X):= \displaystyle \dim \left (\pi_{\mathrm{even}}(X)\otimes \Q \right) -\dim
  \left (\pi_{\mathrm{odd}}(X)\otimes \Q \right) \le 0$.
\item \label{d} $0 \le \chi(X)=\dim H^{\mathrm{even}}(X;\Q)- \dim H^{\mathrm{odd}}(X;\Q)$.
\item \label{e} $\chi(X) >0 \Longleftrightarrow \chi^{\pi}(X)=0$.
\end{enumerate}
%%%%%%%%%%%%%%%%%%
As for the Betti number, the following are known:
\begin{enumerate}
\item \label{poincare} Betti numbers $b_i =\dim H_i(X;\mathbb Q)$ of $X$ satisfy Poincar\'e duality \cite[\S 38  Poincar\'e Duality]{FHT}.  In particular $b_n=1$ and $b_{n-1}=b_1=0$.
\item Betti numbers satisfy inequalities: $b_m \displaystyle  \le {1 \over 2}{n \choose m}, m \ne 0,n$ (cf.  \cite[Corollary to Theorem 1]{Pa}). This inequality implies 
$\dim H^*(X;\mathbb Q) \le 2^{n-1}+1,$
which is sharper than $\dim H^*(X;\mathbb Q) \le 2^n$ (\cite[Theorem 2.75]{FOT}).
\end{enumerate}

Regarding the dimensions of the rational homotopy groups and the rational homology groups (which are the same as the rational cohomology groups by the universal coefficient theorem), 
in \cite{Hil} (cf. \cite{HM, HM2}) M. R. Hilali made the following conjecture, which is well-known as ``Hilali conjecture", which is still open although it is known that the conjecture holds for many spaces such as elliptic spaces of pure type, H-spaces, nilmanifolds, symplectic and cosymplectic manifolds, coformal spaces with only odd-degree generators, formal spaces and hyperelliptic spaces under certain conditions (e.g., \cite{Hil, BFMM, CatMil}): 
\begin{con}[Hilali conjecture]
\begin{equation*}\label{hilali}
\op{dim} \left (\pi_*(X) \otimes \mathbb Q \right ) \le \op{dim} H_*(X;\mathbb Q).
\end{equation*}
\end{con}
For a simply connected rationally elliptic space $X$ we put
$$P_X(t) := \sum_{k \ge 0} \op{dim} H_k(X;\mathbb Q) t^k =  \sum_{k \ge 0} \op{dim} H^k(X;\mathbb Q) t^k,$$
which is the well-known Poincar\'e polynomial of $X$. Similarly we can define the following polynomial
$$P^{\pi}_X(t) := \sum_{k \ge 2} \op{dim} (\pi_k(X)\otimes \mathbb Q) t^k.$$
So, we call 
\begin{itemize}
\item $P_X(t)$ the \emph{(co)homological Poincar\' e polynomial} of $X$,
\item $P^{\pi}_X(t)$ the \emph{homotopical Poincar\' e polynomial} of $X$.
\end{itemize}
%%%%%%%%%
Then the above Hilali conjecture claims the following inequality of these two polynomials at the special value $t=1$:
\begin{equation}\label{hilali-pp}
P^{\pi}_X(1) \le P_X(1).
\end{equation}
%%%%%%%%%
\begin{rem} Here we note the following strict inequalities at the other two special values $t=0, -1$:
\begin{itemize}
\item $0=P^{\pi}_X(0) < P_X(0) =1$.
\item When $t=-1$, we have $P^{\pi}_X(-1) =\chi^{\pi}(X)$ and $P_X(-1) =\chi(X)$. Hence we have
\begin{equation}\label{chi-<}
P^{\pi}_X(-1) < P_X(-1),
\end{equation}
which follows from \ref{c}, \ref{d} and \ref{e} above. Of course  \ref{e} is much sharper than (\ref{chi-<}).
\end{itemize}
\end{rem}
%%%%%%%%%%%
Let $X$ be a quasi-projective algebraic variety.  
Both the homotopy and the cohomology groups carry mixed
Hodge structures (cf. \cite{De1}, \cite{De2}, \cite{Mo},
\cite{Hain 1},\cite{Hain 2}, \cite{Nav}), which are functorial for regular maps, and 
 an invariant of which is given by the generating
functions for the dimensions of graded pieces of Hodge and weight
filtrations as follows:
%%%%%%%%%%%%%%%%%
\begin{equation*}\label{homological}
MH_X(t,u,v) :=
\sum_{k,p,q} \dim \Bigl ( Gr_{F^{\bullet}}^{p} Gr^{W_{\bullet}}_{p+q} H^k (X;\mathbb C)  \Bigr) t^{k} u^p v^q, 
\end{equation*}
where $(W_{\bullet}, F^{\bullet})$ is the mixed Hodge structure of the cohomology groups.
\begin{equation*}
MH^{\pi}_X(t,u,v) :=
\sum_{k,p,q} \dim  \Bigl (Gr_{\tilde F^{\bullet}}^{p} Gr^{\tilde W_{\bullet}}_{p+q} ((\pi_k(X)\otimes \mathbb C)^{\vee})\Bigr ) t^ku^p v^q,
\end{equation*}
%%%%%%%%%%%%%%5
where $(\widetilde W_{\bullet}, \widetilde F^{\bullet})$ is the mixed Hodge structure of the dual of homotopy groups.
They will be called respectively the \emph{cohomological mixed Hodge polynomial} and the \emph{homotopical mixed Hodge polynomial} of $X$.

%%%%%%%%%
If for a simply connected complex algebraic variety $X$ we consider $MH_X(t,u,v)$ and $MH^{\pi}_X(t,u,v)$ for $(u,v)=(1,1)$, then we have
$$P_X(t) = MH_X(t,1,1), \quad P^{\pi}_X(t) = MH^{\pi}_X(t,1,1).$$
Thus the above Hilali conjecture claims 
%%%%%%%%%%%%%%
\begin{equation}\label{hilali-mhs}
MH^{\pi}_X(1,1,1) \leq MH_X(1,1,1).
\end{equation}
%%%%%%%%%%%%%%%

Motivated by inequalities (\ref{hilali-pp}) and (\ref{hilali-mhs}), in 
\cite{Yo,Yo2, LY} one considered the problem of comparison of values of homotopical and
homological Poincar\'e polynomials and corresponding mixed Hodge polynomials for 
values other than $t=1$ of the variables. 
Since the Hilali conjecture clearly becomes true for self-products of any rationally elliptic space taken sufficiently many times, we considered  ``integral stabilization threshold" $\frak{pp}(X;1)$, and, more generally, the integral stabilization threshold
$\frak{pp}(X;\varepsilon)$ defined for $\varepsilon>0$ by
\begin{equation*}
\frak{pp}(X;\varepsilon):= \min\{n(\varepsilon) \, \, | \, \, n P^{\pi}_X(t) < P_X(t)^n, \forall n \ge n(\varepsilon), \forall t \ge \varepsilon  \}.
\end{equation*}
Here $\frak{pp}$ stands for ``Poincar\'e polynomial".
In this paper we give a survey on the properties of this invariant
obtained in \cite{LY}. We also introduce and study the properties of a
new invariant, ``the real
stabilization threshold" $\frak{pp}_{\mathbb R} (X;1)$, 
which can be viewed as a refinement of $\frak{pp}(X;\varepsilon)$.
For $\varepsilon >0$, we let 
\begin{equation*}
\frak{pp}_{\mathbb R}(X;\varepsilon):=\inf \{ r(\varepsilon) \in \mathbb R \, \, | \, \, q P^{\pi}_X(t) < P_X(t)^q, \forall q \ge r(\varepsilon) , \forall t \ge \varepsilon  \}.
\end{equation*}
Here we emphasize that
$\frak{pp}(X;\varepsilon )= \lceil \frak{pp}_{\mathbb R}(X;\varepsilon) \rceil$
where $\lceil x \rceil$ is the ceiling function, i.e., $\lceil x \rceil = \min\{n \in \mathbb Z \, | \, x \le n \}$. Note that the usual Gauss symbol $[x]$ is the floor function $\lfloor x \rfloor :=\max\{n \in \mathbb Z \, | \, n \le x\}$.
It turns out that 
the invariant $\frak{pp}_{\mathbb R}(X;\varepsilon)$ 
can be understood as the ``maximum" value of a real analytic function 
(for more details 
see \S 5, where the case of
  $\frak{pp}_{\mathbb R}(X;1)$ is discussed, so for the general case
  $\frak{pp}_{\mathbb R}(X;1)$ 
can be treated with no substantial changes
). More precisely, we
consider the implicit function $r(t)$ defined by the equation
\begin{equation*}
r(t)P^{\pi}_X(t)=\left (P_X(t) \right )^{r(t)} \quad (t \ge \varepsilon).
\end{equation*}
A Bezout theorem for real analytic curves (cf. \cite{Hov}) yields that
the implicit function $r(t)$ has only finitely many local
maxima-minima. This implicit function can be expressed in terms of the 
classical  Lambert $W$-function $W(z)$, which goes back to Euler and
Lambert (e.g., see \cite{Corless-etal} for the history and some
applications) and is defined as the inverse function given of the
transcendental function $z=we^w$, i.e., $z=W(z)e^{W(z)}$.
More specifically:
\begin{equation*}
r(t)= -\frac{1}{\log P_X(t)} W\left(-\frac{\log P_X(t)}{P^{\pi}_X(t)} \right).
\end{equation*}

We would like to thank Wadim Zudilin  for suggesting a possible relation with Lambert $W$-function and Robert Tijdeman for comments.

We note that the above \emph{stabilization thresholds} can be defined for any two polynomials $P(t)$ and $Q(t)$ with non-negative integers, denoted by $\frak{sth}(P(t),Q(t); \varepsilon)$ and $\frak{sth}_{\mathbb R} (P(t),Q(t); \varepsilon)$, as follows:
\begin{equation*}
\frak{sth}(P(t),Q(t);\varepsilon):= \min\{n(\varepsilon) \, \, | \, \, n P(t) < Q(t)^n, \forall n \ge n(\varepsilon), \forall t \ge \varepsilon  \},
\end{equation*}
\begin{equation*}
\frak{sth}_{\mathbb R}(P(t),Q(t);\varepsilon):=\inf \{ r(\varepsilon) \in \mathbb R \, \, | \, \, q P(t) < Q(t) ^q, \forall q \ge r(\varepsilon) , \forall t \ge \varepsilon  \}.
\end{equation*} 
Here $\frak{sth}$ stands for ``stabilization threshold".

We refer to \cite{Corless-etal} for an account of a large range of problems in which the Lambert $W$-function appears. This account does not include the appearance of $W(z)$ in the context of comparison of linear and exponential growth described in this paper. Such a comparison is certainly an 18th century question and it is natural that it is related to the function in the focus of 18th century mathematics.

Expression of stabilization threshold in terms of values of transcendental functions and results of Baker and Gelfond--Schneider suggest the conjecture of transcendence of our thresholds descried in \S 5. Our results suggest that the collection of thresholds of elliptic rational homotopy types has a natural partition into groups of the types having the same threshold. It would be interesting to understand this distribution better: 
\begin{itemize}
\item What can one say about the collection of thresholds of elliptic rational homotopy types? Assuming the Hilali conjecture, this is a subset of the closed interval $(0,3]$.
\item  What can one say about the set of elliptic rational homotopy types having a fixed threshold $\varepsilon$? 
Note that one has 
$$\frak{pp}_{\mathbb R}(X;\varepsilon)\left (\sum_i \op{rank}\pi_i(X)\varepsilon^i \right ) \le \left (\sum_i \op{rank} H_i(X)\varepsilon^i \right)^{\frak{pp}_{\mathbb R}(X;\varepsilon)}.$$ 
This inequality is \emph{independent of the
inequality in the Hilali conjecture} even for $\varepsilon=1$ in the
sense that Hilali inequality does not provide any information about
the thresholds. In the case when the threshold is equal to 1, one has
inequality for all $t\ge 1$ which does not follows from the inequality for
$t=1$. Rather, the above provides
an additional inequality relation between the sums of ranks of homotopy and
homology groups.
\end{itemize}

The organization of the paper is as follows. In \S 2, \S 3  and \S 4 we recall the results of \cite{LY}; in 
\S 2 we recall the Hilali conjecture and an inequality like the Hilali conjecture for Cartesian self-products of spaces for the homotopical and homological Poincar\'e polynomials and also the homotopical and homological mixed Hodge polynomials; in \S 3 we recall the integral stabilization threshold and compute the thresholds of the spheres $S^{2n}, S^{2n+1}$ and the complex projective spaces $\mathbb CP^n$; in \S 4 we compute the homotopical and homological mixed Hodge polynomials of toric manifolds and their stabilization thresholds. In \S 5 we introduce the real stabilization threshold and discuss its properties. 
In \S 6 we compute the real stabilization thresholds of $S^{2n}, S^{2n+1}$ and  $\mathbb CP^n$.

%%%%%%%%%%%%%
%%%%%%%%%%%%%%%%
\section{Hilali conjecture on products}

First we point out that the inequality $\leq$ in the Hilali conjecture cannot be replaced by the strict inequality $<$. Indeed, the following are well-known results, which follow from Serre Finiteness Theorem \cite{Se}:
%%%%%%%%%%%%%%
$$\pi_i(S^{2k}) \otimes \Q 
=\begin{cases} 
\Q & \, i=2k, 4k-1,\\
\, 0  & \,  i\not =2k, 4k-1,
\end{cases} \quad 
\pi_i(S^{2k+1}) \otimes \Q 
=\begin{cases} 
 \Q & \,  i=2k+1,\\
\, 0 & \, i \not =2k+1.
\end{cases} 
$$
%%%%%%%%%%%%%%%%
Hence we have 
$P^{\pi}_{S^{2k+1}}(t) = t^{2k+1} \, \text{ and} \, P_{S^{2k+1}}(t) = t^{2k+1} +1,$
so
$P^{\pi}_{S^{2k+1}}(t) < P_{S^{2k+1}}(t) \, \ \text{for} \, \,   \forall t.$
For $S^{2k}$ we have 
$P^{\pi}_{S^{2k}}(t) = t^{4k-1} + t^{2k} \, \text{and} \, P_{S^{2k}}(t) = t^{2k} +1,$
thus
$P^{\pi}_{S^{2k}}(1) = P_{S^{2k}}(1) =2.$ 
Therefore the inequality $\leq$ 
cannot be replaced by the strict inequality $<$.

\begin{rem}\label{rem-cpn}
We note that only $S^2$ is a complex manifold and is the $1$-dimensional complex projective space $\mathbb CP^1$. For the $n$-dimensional projective space $\mathbb CP^n$, it follows from the fibration $S^1 \hookrightarrow S^{2n+1} \to \mathbb CP^n$ that the homotopy groups of $\mathbb CP^n$ are
$$\pi_k(\mathbb CP^n) \otimes \Q 
=\begin{cases} 
0 & \, k \not =2,2n+1,\\
\Q & \, k=2, 2n+1.
\end{cases}
$$
Hence we have
$$P_{\mathbb CP^n}(t) = 1 + t^2+ t^4 + \cdots + t^{2n}, \,\, P_{\mathbb CP^n}^{\pi}(t) = t^2 + t^{2n+1}.$$
\end{rem}
%%%%%%%%%%%%
The isomorphisms $\pi_i(X\times Y)=\pi_i(X) \oplus \pi_i(Y)$ and the
K\"unneth formula $H_n(X \times Y, \mathbb Q)= \sum_{i+j=n} H_i(X; \mathbb Q) \otimes H_j(Y; \mathbb Q)$ imply that the homotopical Poincar\'e polynomial $P^{\pi}_X(t)$ and the cohomological Poincar\'e polynomial $P_X(t)$ are respectively additive and multiplicative, i.e., 
$$P^{\pi}_{X \times Y} (t) =P^{\pi}_X(t) + P^{\pi}_Y(t) \quad \text{and} \quad  P_{X \times Y} (t) =P_X(t) \times  P_Y(t).$$
Using these additivity and multiplicativity and some elementary calculus, in \cite{Yo} we show that there exists a positive integer $n_0$
such that for all $n>n_0$ one has 
$$P^{\pi}_{X^n}(1)<P_{X^n}(1)$$
where $X^n= \underbrace{X \times \cdots \times X}_n$ is the Cartesian product of $n$ copies of $X$.
Since $P_X(t)$ and $P_X^{\pi}(t)$ have non-negative coefficients, by the same argument we can show that for any non-negative real number $r$ there exists a positive integer $n_0(r)$
such that for all $n>n_0(r)$ one has $P^{\pi}_{X^n}(r)<P_{X^n}(r)$. Clearly the integer $n_0(r)$ does depend on the choice of the real number. However in \cite{LY} we show the following result (announced in \cite{Yo2}):
%%%%%%%%%%%%%%%%%
\begin{thm}\label{semi-global} Let $X$ be a
simply connected rationally elliptic space. For any positive real number
  $\varepsilon$ there exists
 a positive integer $n(\varepsilon)$ such that for $\forall n \ge
 n(\varepsilon)$ and $\forall t \ge \varepsilon$
\begin{equation}\label{var}
P^{\pi}_{X^n}(t) < P_{X^n}(t). 
\end{equation}
\end{thm} 
%%%%%%%%%%%%
Theorem \ref{semi-global} is an immediate
consequence of the following \cite{LY}:
%%%%%%%%%%%%%5
\begin{lem}\label{general thm} Let $\varepsilon$ be a positive real number. Let $P(t)$ and $Q(t)$ be two polynomials of the following types:
$$P(t) = \sum_{k=2}^p a_kt^k, \,  a_k \ge 0, \quad \quad Q(t) = 1 + \sum_{k=2}^q b_kt^k, \quad b_k \ge 0, \, b_q \not = 0 .$$
(For our purpose it is sufficient to consider $b_q=1$,
  but we do not assume it.) 
Then there exists a positive integer 
$n(\varepsilon)$ such that for $\forall n \ge n(\varepsilon)$
\begin{equation}\label{keyineq}
n P(t) < Q(t)^n \, \, (\forall t \ge \varepsilon).
\end{equation}
\end{lem}
%%%%%%%%%%%
\begin{rem}\label{remark-0} 
\begin{enumerate}
\item Note that, since $X$ is simply connected, $P_X(t)=1$
implies that $X$ is rationally homotopy equivalent to a point
(cf. \cite[Theorem 8.6]{FHT}), and
hence $P_X^{\pi}=0$. In particular, the inequality (\ref{var}) is satisfied 
with $n(\varepsilon)=1,  \forall \varepsilon>0$.  Therefore, in Theorem
\ref{semi-global}  we assume that $P_X(t)>1$. From now on we assume this condition.
\item In the above theorem we cannot let $\varepsilon=0$, in which case there \emph{does not} exist such an integer 
$n(\varepsilon)$, because one may have $\lim_{\varepsilon \to 0}n(\varepsilon)=\infty$.
\end{enumerate}
\end{rem}
%%%%%%%%%%%
In fact the cohomological mixed Hodge polynomial is also multiplicative just like the (cohomological) Poincar\'e polynomial $P_X(t)$:
%%%%%%%%%%%%
\begin{equation*}\label{mh-multi}
MH_{X \times Y}(t,u,v) =MH_X(t,u,v)  \times MH_Y(t,u,v),
\end{equation*}
%%%%%%%%%%%%
which follows from the fact that the mixed Hodge structure is compatible with the tensor product (e.g., see \cite{PS}.)
On the other hand the homotopical mixed Hodge polynomial is additive just like the homotopical Poincar\'e polynomial $P^{\pi}_X(t)$

\begin{equation*}\label{mh-pi-additive}
MH^{\pi}_{X \times Y}(t,u,v) = MH^{\pi}_X(t,u,v) + MH^{\pi}_Y(t,u,v)
\end{equation*}
since $\pi_{*}(X \times Y) = \pi_{*}(X) \oplus \pi_{*}(Y)$ and the category of mixed Hodge structures is abelian and the direct sum of a mixed Hodge structure is also a mixed Hodge structure. 
%%%%%%%%%%%
In this paper the following special multiplicativity and additivity are sufficient:
\begin{equation*}\label{mh-pi-multi}
MH_{X^n}(t,u,v) = \{MH_X(t,u,v)\}^n,
\end{equation*}
\begin{equation*}\label{mh-pi-additive}
MH^{\pi}_{X^n}(t,u,v) = nMH^{\pi}_X(t,u,v). \quad
\end{equation*}
%%%%%%%%%%%%%%

A ``mixed Hodge polynomial" version of 
Theorem \ref{semi-global} for algebraic varieties
is the following \cite{LY,Yo2}):
%%%%%%%%%%%%5
\begin{thm}\label{mhp}
 Let $0< \varepsilon \ll 1$  and $r>0$ be positive real numbers and let $\mathscr C_{\varepsilon,r}:=[\varepsilon, r] \times [\varepsilon,r] \times [\varepsilon,r] \subset (\mathbb R_{\ge 0})^3$ be the cube of size $r - \varepsilon$.
Then there exists a positive integer $n_r$ such that for 
$\forall n \ge n_r$
the following strict inequality holds
$$MH^{\pi}_{X^n}(t,u,v) <  MH_{X^n}(t,u,v)$$
for $\forall (t,u,v) \in \mathscr C_{\varepsilon,r}.$ 
\end{thm}
%%%%%%%%%%%%%%%%
%%%%%%%%%%%%
%%%%%%%%%%%%%%%%%%%%%%%
\section{Integral stabilization threshold $\frak{pp}(X;\varepsilon)$}

Theorem \ref{semi-global} suggests the following invariant of a rationally elliptic
homotopy type:
%%%%%%%%%%%%%%
\begin{defn}\label{threshold} 
\begin{enumerate}
\item \emph{The integral stabilization 
threshold} or simply \emph{the integral threshold}, denoted by $\frak{pp}(X; \varepsilon)$, is the smallest integer
  $n(\varepsilon)$
such that inequality (\ref{var}) takes place for $\forall n \ge n(\varepsilon)$.  
\item The smallest integer
  $n(\varepsilon)$
such that inequality (\ref{keyineq}) takes place for $\forall n \ge n(\varepsilon)$ is denoted by $\frak{sth}(P(t),Q(t);\varepsilon)$.
\end{enumerate}
\end{defn}
%%%%%%%%%%%
With the Hilali conjecture in mind, we consider $\varepsilon =1$, thus we consider $\frak{pp}(X; 1)$.
%%%%%
\begin{ex} $\frak{pp}(S^{2n+1};1)=1$. Because $P^{\pi}_{S^{2k+1}}(t) = t^{2k+1}$, $P_{S^{2k+1}}(t) = t^{2k+1} +1$ and we have $n(t^{2k+1}) < (t^{2k+1} +1)^n$ for $\forall n \ge 1$ and for $\forall t \ge 0$.
\end{ex}
%%%%%%%%%%
%%%%%%%%%%
\begin{ex}\label{s-2n} $\frak{pp}(S^{2n};1)=3$.
Recall that $P_{S^{2n}}(t) = 1 + t^{2n}$ and $P_{S^{2n}}^{\pi}(t) = t^{2n} + t^{4n-1}$.
It is easy to see that if $t \geq 1$, 
$2(t^{2n}+t^{4n-1}) < (1+t^{2n})^2$ has the only exception for $t=1$.
But we see that 
$3(t^{2n}+t^{4n-1}) < (1+t^{2n})^3$ for $\forall t \geq 1$, then by induction on the power $m$, we have  
$m(t^{2n}+t^{4n-1}) < (1+t^{2n})^m$ for $\forall m \geq 3$.
Therefore we have $\frak{pp}(S^{2n};1)=3.$ 
\end{ex}
%%%%%%%%%%
%%%%%%%%%%%%
\begin{ex}\label{ex-cpn}
$$\frak{pp}(\mathbb CP^n;1)= \begin{cases} 3, & n=1, \\
2, & n \ge 2.
\end{cases}$$
As in Remark \ref{rem-cpn}, we have
$P_{\mathbb CP^n}(t)= 1+ t^2 + \cdots +t^{2n}, \, P_{\mathbb CP^n}^{\pi}(t)= t^2 + t^{2n+1}.$
\begin{enumerate}
\item $\frak{pp}(\mathbb CP^1;1)=3$: This follows from 
Example \ref{s-2n} since $\mathbb CP^1=S^2$.
%%%%%%%%%%%%%%
\item $\frak{pp}(\mathbb CP^n;1)=2$ for $n \ge 2$:
Clearly, if $t \geq 1$, $t^2+t^{2n+1} \not < 1+t^2+\cdots +t^{2n}$, but we do have 
$2(t^2+t^{2n+1}) <(1+t^2+\cdots +t^{2n})^2.$
Indeed,
\begin{align*}
(1+t^2+\cdots +t^{2n})^2- 2(t^2+t^{2n+1}) &
\geq \left \{(1+t^2)+ t^{2n} \right \}^2- 2(t^2+t^{2n+1})\\
& > 2t^{2n}(1+t^2) - 2(t^2+t^{2n+1}) \\
& = 2t^{2n+1}(t-1) +2t^2(t^{2n-2}-1) \\
& \geq 0 \quad \text{(since $t \geq 1$ and $2n-2\geq 2$)}
\end{align*}

Then, as in Example \ref{s-2n}, by induction on the power $m$, 
for $\forall m \geqq 2$ we have 
$m(t^2+t^{2n+1}) <(1+t^2+\cdots +t^{2n})^m.$
Therefore we have $\frak{pp}(\mathbb CP^n;1)=2.$ 
\end{enumerate}
\end{ex}
%%%%%%%%%
Thus, for $X = S^{2n+1}, \, S^{2n}, \, \mathbb CP^n$, we have $\frak{pp}(X;1) \le 3$.
%%%%%%%%%%%%%%%%%%%
It turns out that this inequality always holds \emph{provided} the Hilali conjecture holds, i.e., $P^{\pi}_X(1) \le P_X(1)$ holds:
\begin{thm}[\cite{LY}] If the Hilali conjecture is correct, then for any simply connected rationally elliptic space $X$ we have
$\frak{pp}(X;1) \le 3.$
\end{thm}
%%%%%%%%%%%
If we do not use the Hilali conjecture, we can show the following:
%%%%%%%%%
\begin{thm}[\cite{LY}]  Let $X$ be a simply connected rationally elliptic
space of formal dimension $n \ge 3$. Then   
$\frak{pp}(X;1) \le n.$
\end{thm}
%%%%%%%%%%%%%%

\begin{rem}\label{form-product} We have the following inequality 
if $P_X(\varepsilon) \ge 2$ and $P_Y(\varepsilon) \ge 2$: 
\begin{equation}\label{pp-product}
 \frak{pp}(X\times Y;\varepsilon) \le \op{max} \{\frak{pp}(X;\varepsilon),\frak{pp}(Y;\varepsilon) \}
\end{equation}
In particular, since $P_X(1) \ge 2$ and $P_Y(1) \ge 2$ (see Remark \ref{remark-0} (1)), we have
\begin{equation*}
 \frak{pp}(X\times Y;1) \le \op{max} \{\frak{pp}(X;1),\frak{pp}(Y;1) \}.
\end{equation*}
\end{rem}
%%%%%%%%%%%%
Similarly to $\frak{pp}(X;\varepsilon)$, 
we define the following:
%%%%%%%%%%%%
\begin{defn}\label{mhp-defn}
The smallest integer $n_0$ such that for $\forall n \ge n_0$ the following holds
\begin{equation*}
MH^{\pi}_{X^n}(t,u,v) < MH_{X^n}(t,u,v)  \quad \forall t \ge a, \forall u \ge b, \forall v \ge c.
\end{equation*}
is denoted by $\frak{mhp}(X; a,b,c)$. 
Here $\frak{mhp}$ stands for ``mixed Hodge polynomial".
\end{defn}
%%%%%%%%%%%%%
\begin{rem} In a similar manner to the proof of (\ref{pp-product}) in Remark \ref{form-product}, we can see the following inequality as to the threshold $\frak{mhp}$:
\begin{equation*}
 \frak{mhp}(X\times Y;a, b, c) \le \op{max} \{\frak{mph}(X;a,b,c),\frak{mhp}(Y;a,b,c) \}
\end{equation*}
for positive real numbers $a, b, c$ such that $MH_X(a,b,c) \ge 2$ and $MH_Y(a,b,c) \ge 2$.
\end{rem}
%%%%%%%%%%%%%
\begin{rem}
It follows from Theorem \ref{mhp} that for any $\varepsilon >0$ and $r > \varepsilon$ there exists the smallest integer $n_{\varepsilon, r}$ such that for $\forall n \ge n_{\varepsilon, r}$
$$MH^{\pi}_{X^n}(t,u,v) < MH_{X^n}(t,u,v)  \quad \forall (t,u,v) \in \mathscr C_{\varepsilon,r}=[\varepsilon, r] \times  [\varepsilon, r] \times [\varepsilon, r].$$
This smallest integer $n_{\varepsilon, r}$ is denoted by  $\frak{mhp}(X; [\varepsilon, r] ,[\varepsilon, r] ,[\varepsilon, r])$.
\end{rem}

%%%%%%%%%%%%%%%%%%%
\section{Toric manifolds and $\frak{mhp}(X; 1,1,1)$} 

%%%%%%%%%%%
\subsection{$\C^{n+1}\setminus 0$}\label{complementtopoint} 

This is a smooth quasi-projective variety, for which 
the mixed Hodge structures on cohomology and homotopy 
can be constructed using log-forms (cf. \cite{De1} and \cite{Mo} resp.).
Since this space can be retracted on $S^{2n+1}$ and the Hurewicz
isomorphism preserves the Hodge structure (cf. \cite{Hain 1}) and
calculating the mixed Hodge structure on $H_n(\C^{n+1}\setminus 0)$ 
(for example using  Gysin exact sequence for the homology of the
complement to smooth divisor on the blow up of $\PP^{n+1}$ at a point)
we obtain:
\begin{equation*} MH_{\mathbb C^{n+1} \setminus
    \{0\}}(t,u,v) = 1 + t^{2n+1}(uv)^{n+1},
\end{equation*}
\begin{equation*} MH^{\pi}_{\mathbb C^{n+1} \setminus
  \{0\}}(t,u,v) = t^{2n+1}(uv)^{n+1}. \quad \, \,\, 
\end{equation*}
Hence we have
\begin{equation*}MH_{\mathbb C^{n+1} \setminus \{0\}}(t,u,v) = 1 +
  MH^{\pi}_{\mathbb C^{n+1} \setminus \{0\}}(t,u,v).
\end{equation*}
%%%%%%%%%%%%%
%%%%%%%%%%%%%%%%
\subsection{Projective spaces}
The mixed Hodge polynomials of the projective space $\mathbb CP^n$ are as follows:
\begin{equation}\label{cpn-pi}
MH_{\mathbb CP^n}(t,u,v) = 1 + t^2uv + t^4(uv)^2 +
  \cdots +t^{2i}(uv)^i + \cdots + t^{2n}(uv)^n.
\end{equation}
\begin{equation}\label{cpn-homo}
MH_{\mathbb CP^n}^{\pi}(t,u,v) = t^2uv +
  t^{2n+1}(uv)^{n+1}. \hspace{5cm} 
\end{equation}
The cohomological case is trivial and the claim in the homotopical case
follows using the Hurewicz isomorphism for $\pi_2$ and for higher
homotopy groups the locally trivial fibration $\mathbb C^{\times}
\hookrightarrow \mathbb C^{n+1} \setminus \{0\} \to \mathbb CP^n$, the
calculation in \S \ref{complementtopoint} and the corresponding exact sequence
$$ \cdots \to \pi_{2n+1}(\mathbb C^{\times}) \to \pi_{2n+1}(\mathbb C^{n+1} \setminus \{0\}) \to \pi_{2n+1}(\mathbb CP^n) \to \pi_{2n}(\mathbb C^{\times}) \to \cdots$$
which is an exact sequence of mixed Hodge structures \cite[Theorem 4.3.4]{Hain 1}. 
%%%%%%%%%%%
Then as in Example \ref{ex-cpn} one easily verifies that
\begin{equation*}
\frak{mhp}(\mathbb CP^n;1,1,1) = 
\begin{cases} 3, & n =1,\\
2, & n \ge 2.
\end{cases}
\end{equation*}
In fact, $\frak{mhp}(\mathbb CP^n;1,1,1)=2$ can be made to the following a bit sharper statement: for $\forall m \ge 2$
$$MH^{\pi}_{(\mathbb CP^n)^m}(t,u,v) < MH_{(\mathbb CP^n)^m}(t,u,v)  \quad \text{for} \, \, \forall t \ge 1, \forall (u,v) \,\, \text{such that} \,\, uv \ge 1.$$
%%%%%%%%%%%%%%%
%%%%%%%%%%%%%%%%
%%%%%%%%%%%%%%%
\subsection{Compact toric manifolds}
In \cite[Theorem 3.3]{BMM} I. Biswas, V. Mu\~noz and A. Murillo show that the
homological Poincar\'e polynomial of a rationally elliptic toric
manifold coincides with that of a product of complex projective
spaces. Below, using a recent result due to M. Wiemeler \cite{W} we show that the same thing holds for the homotopical Poincar\'e polynomial, in fact, for the homotopical mixed Hodge polynomial, and furthermore we also show that the homological mixed Hodge polynomial of a rationally elliptic toric manifold coincides with that of a product of complex projective spaces, which is a stronger version of the above result of Biswas--Mu\~noz--Murillo:
%%%%%%%%%%%%%%%%%%%
\begin{thm}[\cite{LY}] \label{toric MHP} The homotopical and cohomological  mixed Hodge polynomials of a rationally elliptic toric manifold of complex dimension $n$ coincides with 
those of a product of complex projective spaces. To be more precise, if $X$ is the quotient of
$$\prod_{i=1}^k (\mathbb C^{n_i+1} \setminus \{0\})$$
by a free action of commutative algebraic groups, i.e., $(\mathbb C^{\times})^k$. Here $n = \sum_{i=1}^k n_i.$ Then we have
\begin{enumerate}
\item  $MH^{\pi}_X(t,u,v)=MH^{\pi}_{\prod_i^k \mathbb CP^{n_i}} (t,u,v) = \sum_{i=1}^k MH^{\pi}_{\mathbb CP^{n_i}}(t,u,v)$, i.e.,
$$MH^{\pi}_X(t,u,v) = \sum_{i=1}^k \Bigl (t^2uv+t^{2n_i+1}(uv)^{n_i+1} \Bigr) = kt^2uv + \sum_{i=1}^k t^{2n_i+1}(uv)^{n_i+1}.$$
\item $MH_X(t,u,v)=MH_{\prod_i^k \mathbb CP^{n_i}} (t,u,v) = \prod_{i=1}^k MH_{\mathbb CP^{n_i}}(t,u,v)$, i.e.,
$$MH_X(t,u,v) = \prod_{i=1}^k \Bigl (1 + t^2uv + \cdots + t^{2j}(uv)^j + \cdots + t^{2n_i}(uv)^{n_i} \Bigr ).$$
\end{enumerate}
\end{thm}
%%%%%%%%%%%%%%%%%
It follows from the above Theorem \ref{toric MHP}
that the cohomological and homotopical Poincar\'e polynomials of a
rationally elliptic toric manifold are the same as those of a product of complex projective spaces. In particular, the Hilali conjecture, which for toric
varieties follows immediately as a consequence of formality (cf. \cite{HM}),
can be checked by a direct calculation.
%%%%%%%%%%%%%%%
\begin{cor}\label{cor-toric} Let $X$ be a rationally elliptic toric manifold and let
$$MH_X(t,u,v)=MH_{\prod_i^k \mathbb CP^{n_i}} (t,u,v) ,\qquad MH^{\pi}_X(t,u,v)=MH^{\pi}_{\prod_i^k \mathbb CP^{n_i}} (t,u,v).$$
If each $n_i \ge 2$, then $\frak{mhp}(X;1,1,1)=2$, and if $n_i=1$ for some $i$, then $\frak{mhp}(X;1,1,1)=3.$
\end{cor}
%%%%%%%%%\
%%%%%
\begin{rem}\label{rem-toric} Even if we fix $u=1$ and $v=1$ in the above proof of Corollary \ref{cor-toric}, we have the same proof, therefore we have that
if each $n_i \ge 2$, then $\frak{pp}(X;1)=2$, and if $n_i=1$ for some $i$, then $\frak{pp}(X;1)=3.$
\end{rem}
%%%%%%%%%%%%%
%%%%%%%%%%%%%%%
\subsection{Arrangements of linear subspaces} G. Debongnie
(cf. \cite{deb}) described the structure of arrangements of subspaces
in $\C^n$ whose complements are rationally elliptic. If follows that
such complements are products of $\prod_i \left (\C^{n_i+1}\setminus 0 \right )$.
Combining this with  calculation in \S \ref{complementtopoint},  we obtain:
%%%%%%%%%%%%%
\begin{thm}[\cite{LY}] \label{arrangements} The homotopical and cohomological mixed Hodge
  polynomials of a simply connected rationally elliptic complement $X$
  of an arrangement of linear subspaces are as follows:
\begin{enumerate}
\item  
$MH^{\pi}_X(t,u,v) = MH^{\pi} _{\prod^k_i \left (\C^{n_i+1}\setminus 0 \right )} (t,u,v) =
\sum_{i=1}^k MH^{\pi}_{\C^{n_i+1}\setminus 0}(t,u,v)$, i.e.,  
$$ MH^{\pi}_X(t,u,v) = \sum_{i=1}^k t^{2n_i+1}(uv)^{n_i+1}.$$
\item $MH_X(t,u,v)=MH_{\prod_i^k \left (\C^{n_i+1}\setminus 0 \right )} (t,u,v) = \prod_{i=1}^k MH_{\C^{n_i+1}\setminus 0 }(t,u,v)$, i.e.,
$$MH_X(t,u,v) = \prod_{i=1}^k \Bigl (1 + t^{2n_i+1}(uv)^{n_i+1} \Bigr ).$$
\end{enumerate}
\end{thm}
%%%%%%%%%%%
%%%%%%%%%%%%%%%%
In particular, we obtain
\begin{cor}\label{cor-arrangements} 
$\frak{pp}(X;1)=1$ and $\frak{mhp}(X; 1,1,1)=1.$
\end{cor}
%%%%%%%%%%%%%%%%%%
%%%%%%%%%%%%%%%%%
\section{Real stabilization threshold $\frak{pp}_{\mathbb R}(X;1)$ and Hovanski\u{\i}'s theorem}

The integral stabilization threshold $\frak{pp}(X;1)$ is 
the smallest \emph{integer} $n_0$ such that for $\forall n \ge n_0$ the following inequality holds
$$ P^{\pi}_{X^n}(t) <P_{X^n}(t) \quad \forall t \ge 1.$$
In other words this inequality holds for the product space $X^{\frak{pp}(X;1)}$. 
On the other hand this inequality is the same as
$$nP^{\pi}_X(t) < \left (P_X(t) \right)^n \quad \forall t \ge 1,$$
which is a key ingredient for the results obtained so far.
When it comes to this expression, $n$ does not have to be an integer, but can be a positive real number. Thus we can define the following:
%%%%%%%%%%
%%%%%%%%%%%%%%%
\begin{defn}\label{real-pp} Let $\varepsilon >0$.
\begin{enumerate}
\item  The 
real stabilization threshold $\frak{pp}_{\mathbb R}(X;\varepsilon)$ of a simply connected elliptic space $X$ is defined by
\begin{equation*}
\frak{pp}_{\mathbb R}(X;\varepsilon):=\inf \{r \in \mathbb R \, \, | \, \, q P^{\pi}_X(t) < P_X(t)^q, \forall q \ge r, \forall t \ge \varepsilon  \}.
\end{equation*}
\item For two polynomials $P(t)$ and $Q(t)$ with non-negative integral coefficients, the real stabilization threshold $\frak{sth}_{\mathbb R}(P(t),Q(t);\varepsilon)$ of $P(t)$ and $Q(t)$  is defined by
\begin{equation*}
\frak{sth}_{\mathbb R}(P(t),Q(t);\varepsilon):=\inf \{r \in \mathbb R \, \, | \, \, q P(t) < Q(t)^q, \forall q \ge r, \forall t \ge \varepsilon  \}.
\end{equation*}
\end{enumerate}
\end{defn}
%%%%%%%
\begin{rem} 
As we remark in the Introduction, 
$\frak{pp}(X;\varepsilon )= \lceil \frak{pp}_{\mathbb R}(X;\varepsilon) \rceil$
where $\lceil x \rceil$ is the ceiling function,  
i.e., $\lceil x \rceil = \min\{n \in \mathbb Z \, | \, x \le n \}$.
Similarly, we have $\frak{sth}(P(t),Q(t); \varepsilon )= \lceil \frak{sth}_{\mathbb R}(P(t),Q(t); \varepsilon ) \rceil$.
\end{rem}
%%%%%%%%%%%
In this paper we consider the case when $\varepsilon=1$, i.e., $\frak{pp}_{\mathbb R}(X;1)$. 
Unlike the integral stabilization threshold $\frak{pp}(X;1)$, the real stabilization threshold $\frak{pp}_{\mathbb R}(X;1)$ is much harder to analyze and even in the cases of $X=S^{2n+1}, S^{2n}, \mathbb CP^n$ it is quite difficult, as we see below.

In order to understand the real stabilization threshold better, we use a theorem due to A. G. Hovanski\u{\i} \cite{Hov}.
First we recall the definitions of \emph{a Pfaffian chain (P-chain)} and \emph{a P-system} from \cite{Hov}.
%%%%%%%%%%%%
\begin{defn}(a Pfaffian chain) We say that analytic functions $f_1, f_2, \cdots, f_k$ on $\mathbb R^n$ form \emph{a Pfaffian chain (P-chain)} of length $k$ if all partial derivatives of each $f_j$ in the chain $\{f_1, f_2, \cdots, f_k\}$ are expressible as polynomials of the first $j$ functions of the chain and the coordinate functions $\mathbb R^n$ . In other words, for all $i$ such that $1 \le i \le n$ and all $j$ such that $1 \le j \le k$ there exist polynomials $P_{ij}(x_1, \cdots, x_n, u_1, \cdots, u_j)$ such that
$$ \frac{\partial f_j}{\partial x_i}(x_1, \cdots, x_n) = P_{ij}(x_1, \cdots, x_n, f_1, \cdots, f_j).$$
\end{defn}
%%%%%%%%%%%%
\begin{defn}(a $P$-system) A $P$-system in $\mathbb R^n$  is any system $Q_1 = \cdots = Q_m$ of equations in which each $Q_p$ are polynomials of coordinate functions in $\mathbb R^n$ and functions of a $P$-chain. The \emph{complexity} of a $P$-system is the following collection of numbers: 
$$\text{$n$, the length $k$ of the $P$-chain, and the degrees of the polynomials $Q_p$ and $P_{ij}$}.$$
\end{defn}
%%%%%%%%%%%
Here is a theorem due to Hovanski\u{\i} \cite[Theorem 1]{Hov}:
\begin{thm}[Hovanski\u{\i}'s theorem]\label{hov}
The number of nondegenerate roots of a $P$-system consisting of $n$ equations in $\mathbb R^n$ is finite and bounded from above by an explicitly given function of the complexity of the $P$-system.
\end{thm}
%%%%%%%%%%%%%%
\begin{lem} Let $Q(s)$ be a polynomial. Then the following functions on $\mathbb R^2$ with coordinates $s$ and $r$ form a Pfaffian chain.
$$f_1(s,r):=\frac{1}{Q(s)}, \quad f_2(s,r):=\log Q(s), \quad f_3(s,r):=Q(s)^r.$$
\end{lem}
\begin{proof} It is straightforward, but for the sake of convenience of the reader, we check it. Indeed the derivatives become as follows:
\begin{enumerate}
\item 
\begin{enumerate}
\item $\frac{\partial f_1}{\partial s} = -Q(s)^{-2}Q'(s) = - f_1^2 Q'(s)$. Hence, by letting $P_{11}(s,r,u_1):=-u_1^2Q'(s)$,
we have $\frac{\partial f_1}{\partial s}=P_{11}(s,r,f_1)$. 
\item $\frac{\partial f_1}{\partial r} =0$. Hence, by letting $P_{12}(s,r,u_1)=0$, we have $\frac{\partial f_1}{\partial r} =P_{12}(s,r,f_1)=0$
\end{enumerate}
\item 
\begin{enumerate}
\item $\frac{\partial f_2}{\partial s} = \frac{Q'(s)}{Q(s)}= f_1Q'(s)$. Hence, by letting $P_{}(s,r,u_1,u_2):=u_1Q'(s)$,
we have $\frac{\partial f_2}{\partial s}=P_{11}(s,r,f_1,f_2)$. 
\item
$\frac{\partial f_2}{\partial r} =0$. Hence, by letting $P_{22}(s,r,u_1,u_2)=0$, we have $\frac{\partial f_2}{\partial r} =P_{12}(s,r,f_1,f_2)=0$.
\end{enumerate}
\item 
\begin{enumerate}
\item $\frac{\partial f_3}{\partial s} = rQ(s)^{r-1}Q'(s)= rf_3f_1Q'(s)$. Hence, by letting $P_{31}(s,r,u_1,u_2,u_3):=ru_1u_3Q'(s)$,we have $\frac{\partial f_3}{\partial s}=P_{31}(s,r,f_1,f_2,f_3)$. 
\item 
$\frac{\partial f_3}{\partial r} =Q(s)^r \log Q(s)=f_3 f_2$. Hence, by letting $P_{32}(s,r,u_1,u_2,u_3)=u_2u_3$, we have $\frac{\partial f_3}{\partial r} =P_{32}(s,r,f_1,f_2,f_3)=0$.
\end{enumerate}
\end{enumerate}
\end{proof}
%%%%%%%%%%%%%%
\begin{lem} Let $P(s)$ and $Q(s)$ be polynomials and let $R(s,r):=rP(s) -Q(s)^r.$ Let $f_1,f_2,f_3$ be the above Pfaffian chain. Then the following system is a P-system:
\begin{equation}\label{p-system}
R(s,r)=\frac{\partial R}{\partial s} =0.
\end{equation}
\end{lem}
\begin{proof}
$$R(s,r) = rP(s) -f_3,$$
$$\frac{\partial R}{\partial s}= rP'(s) -rQ(s)^{r-1}Q'(s) = rP'(s) -rf_3f_1Q'(s).$$
Hence, by letting $Q_1(s,r,u_1,u_2,u_3):=rP(s)-u_3, Q_2(s,r,u_1,u_2,u_3):=rP'(s) -ru_1u_3Q'(s)$, we have
$R(s,r)=Q_1(s,r,f_1,f_2,f_3)$ and $\frac{\partial R}{\partial s}=Q_2(s,r,f_1,f_2,f_3).$ Therefore (\ref{p-system}) is a $P$-system.
\end{proof}
%%%%%%%%%
Hence, from the above Hovanski\u{\i}'s theorem we get the following corollary:
%%%%%%%%
\begin{cor}\label{cor-1} The above $P$-system $R(s,r)=\displaystyle \frac{\partial R}{\partial s} =0$ has only finitely many solutions.
\end{cor}
%%%%%%%%
The solutions $(s,r)$ of $R(s,r)=0$ are of the following types:
For a fixed $s$, the solutions $r$ of $R(s,r)=0$ are the intersection of the straight line $y=P(s)x$ and the exponential function $y=Q(s)^x$, hence we have three cases
\begin{enumerate}
\item there are exactly two different solutions $r_1(s)$ and $r_2(s)$ in the case when they intersect at two different points $(s,r_1(s))$, $(s,r_2(s))$,
\item there is just one solution $r_0(s)$ in the case when the line is tangent to the exponential curve at the point $(s,r_0(s))$,
\item there is no solution in the case when they do not intersect.
\end{enumerate}
%%%%%%%%%%%%%
So we define the implicit function $r(s)$ of the equation $R(s,r(s))=0$ by
\begin{equation*}
r(s) = \begin{cases}
\op{max}\{r_1(s), r_2(s)\}, &\text{if they intersect at two different points $(s,r_1(s))$, $(s,r_2(s))$,}\\
r_0(s), & \text{if the line is tangent to the exponential curve} \\
& \hspace{6cm}\text{at one point $(s,r_0(s))$,}\\
\text{not defined}, & \text{if they do not intersect.}\\
\end{cases}
\end{equation*}
From the above Corollary \ref{cor-1} we get the following corollary:
\begin{cor}\label{f-local} Let the situation be as above. The implicit function $r(s)$ defined by $R(s,r(s))=0$ has finitely many local maxima-minima.
\end{cor}
%%%%%%%%%%
\begin{proof} If we consider the derivative $\frac{dr}{ds}(s_0)$ of the implicit function $r=r(s)$ at $s_0$ with $r_0=r(s_0)$ such that $R(s_0,r_0)=0$, we do have
$$\frac{\partial R}{\partial s}(s_0,r_0) + \frac{\partial R}{\partial r}(s_0,r_0) \frac{dr}{ds}(s_0) =0.$$
So, if $r(s)$ has a local maximum or minimum at $s_0$, then $\frac{dr}{ds}(s_0)=0$, thus $\frac{\partial R}{\partial s}(s_0,r_0)=0$. 
Namely, $(s_0,r_0)$ has to be a solution of the above $P$-system. Therefore, the implicit function $r(s)$ defined by $R(s,r(s))=0$ has only finitely many local maxima-minima. 
\end{proof}
%%%%%%%%%%%%

Before going further on, we point out that the above implicit function $r(s)$ can be described in terms of the Lambert $W$-function. The inverse function of $z=we^w$ is called \emph{the Lambert $W$-function} (cf. \cite{Corless-etal} for the history and some
applications) and denoted by $W(z)$, thus it satisfies
\begin{equation*}\label{trans-eq}
z = W(z)e^{W(z)}.
\end{equation*}
Let $\widetilde W(z)$ be the inverse function of $z={e^w \over w}$, which
satisfies   
\begin{equation*} z={e^{\widetilde W (z)} \over {\widetilde W(z)}}.
\end{equation*}
This function $\widetilde W(z)$ is a specialization of 
\emph{the generalized Lambert $W$-function}
proposed in \cite{Scott-etal} and studied further in
\cite[\S 3]{JMM}.
The implicit function $r(s)$ of our equation
\begin{equation*}
r(s)P(s) = Q(s)^{r(s)}
\end{equation*}
is explicitly expressed by using the above function $\widetilde
W(z)$. Specifically, one has
\begin{align*}
P(s) & = \frac{Q(s)^{r(s)}}{r(s)}\\
& = \frac{e^{(\log Q(s))r(s)}}{r(s)}.
\end{align*}
From which we have
$$\frac{P(s)}{\log Q(s)}= \frac{e^{(\log Q(s))r(s)}}{(\log Q(s))r(s)}.$$
Hence we have
$$\widetilde W \left (\frac{P(s)}{\log Q(s)} \right ) = (\log Q(s))r(s).$$
Therefore the implicit function $r(s)$ is obtained as
$$r(s) = \frac{1}{\log Q(s)} \widetilde W \left (\frac{P(s)}{\log Q(s)} \right ) .$$

Finally we notice that $z =W(z)e^{W(z)}$ implies that
$$- \frac{1}{z}= W \left(-\frac{1}{z} \right) e^{W \left(-\frac{1}{z} \right) }, \quad \text{i.e.,} \quad z= \frac{e^{-W \left(-\frac{1}{z} \right) }}{- W \left(-\frac{1}{z} \right) }.$$
Hence we have
$$\widetilde W(z) = - W \left(-\frac{1}{z} \right).$$
Therefore we have that
$$r(s) = -\frac{1}{\log Q(s)} W \left (-\frac{\log Q(s)}{P(s)} \right ).$$ 

%%%%%%%%%%%%%%%%%%%%
So far there is no assumption on the coefficients of the polynomials $P(s)$ and $Q(s)$. From here on, we assume that they are respectively the homotopical Poincar\'e polynomial and homological Poincar\'e polynomial of a simply connected elliptic space, in particular we have that $Q(s)=1 +\cdots.$ 
%%%%%%%%%
\begin{pro} Let $P(s)$ and $Q(s)$ be as above. The above implicit function $r=r(s)$ of the equation 
$rP(s) =Q(s)^r$ is a bounded function.
\end{pro}
%%%%%%%%%
\begin{proof}  First we point out that the domain of the implicit function $r=r(s)$ is not necessarily the whole half interval $[1, \infty)$, thus there are some intervals of $[1, \infty)$, where the implicit function is not defined, in other words, if we denote the domain of the implicit function by $\mathcal D(r(s))$, then it is possible that $\mathcal D(r(s)) \subsetneqq [1,\infty)$. However, \emph{even if there are some interval where the implicit function is not defined, we still consider the function is bounded where it is not defined.} It follows from the proof of \cite[Lemma 2.1]{LY} that there exists $s_0 \ge 1$ and an integer $N_0$ such that
\begin{equation*}
rP(s) < Q(s)^r \quad \text{for} \, \forall s \ge s_0, \forall r \ge N_0 .
\end{equation*}
Thus the graph of the implicit function $r=r(s)$ does not intersect the translated quadrant $\{(s,r) \, | \, s \ge s_0, r \ge N_0\}$, in other words, the implicit function $r=r(s)$ is bounded above by $N_0$ on the half interval $[s_0, \infty)$. Now we want to show that it is also bounded on the closed interval $[1,s_0]$ as well. Now let us consider the implicit function $\widetilde r (x,y)$
of two variables $x,y$, given by the largest solution to 
$rx=y^r$. Here we note that
%%%%%%%
\begin{equation}\label{relation-r}
r(s)=\widetilde r(P(s),Q(s)).
\end{equation}
%%%%%%%%
We let 
$$m_P:=\min_{1 \le s \le s_0}P(s), M_P:=\max_{1 \le s \le s_0}P(s), \quad m_Q:=\min_{1 \le s \le s_0}Q(s), M_Q:=\max_{1 \le s \le s_0}Q(s).$$
Then we consider
a compact rectangle 
$$\mathcal R_{m,M}:=\{(x,y) |\, \, m_P\le x \le M_P, m_Q \le y \le M_Q\}.$$
The claim is that $\widetilde r(x,y)$ is defined on a compact subset of this
rectangle $\mathcal R_{m,M}$, which implies that 
$\widetilde r(x,y)$ is bounded, since it is a continuous function. Therefore it follows from (\ref{relation-r}) that the implicit function $r=r(s)$ is also bounded on $[1,s_0]$.
 Since
the domain $\mathcal D$ of $\widetilde r(x,y)$ is a subset of the compact bounded set $\mathcal R_{m,M}$, it is enough to show that domain $\mathcal D$ is
closed. Let $\phi(y)$ be the slope of the line in $(u,v)$-plane 
containing the origin and tangent to the curve given by $u=y^v$.
$\phi(y)$ is of course computed as follows:
The tangent line of the curve $u=y^v$ at a point $(v_0,u_0=y^{v_0})$ is given by
$u-y^{v_0} = y^{v_0}\log y (v -v_0).$
If this line goes through the origin, then it follows from $-y^{v_0} = y^{v_0}\log y \, (-v_0)$ that 
$v_0 \log y =1,$ hence $v_0= \frac {1}{\log y}$. Hence the slope $\phi(y)$ is
$$\phi(y) = y^{\frac {1}{\log y}}\log y = e \log y.$$
Thus $\phi(y)$  is a continuous function for $y > 1$. 
We assume that $Q(s) = 1 + \cdots$ and is not a constant polynomial, 
hence $\displaystyle \min_{1 \le s \le s_0}Q(s)=m_Q >1$.
The domain of the function $\widetilde r(x,y)$ considered on the rectangle $\mathcal R_{m,M}$ 
is the subset given by $x \ge \phi(y)$, i.e., $\mathcal D:=\{(x,y) \, | \, \, (x,y) \in \mathcal R_{m,M}, x \ge \phi(y) \}$.
 Continuity of $\phi$ implies that the domain $\mathcal D$
is closed. \\
\end{proof}
%%%%%%%%%%%%%
\begin{pro}\label{pro-r(s)} Let the situation be as above.
The finite critical points $s$ of the implicit function $r=r(s)$ satisfy the following equation
\begin{equation}\label{equ-0}
\frac{P'(s)Q(s)}{Q'(s)}= Q(s)^{\frac{P'(s)Q(s)}{P(s)Q'(s)}}.
\end{equation}
\end{pro}
%%%%%%%%%%%%
\begin{proof}
Let $r=r(s)$ be the implicit function of the equation $rP(s)=Q(s)^r$.
Differentiating $rP(s)=Q(s)^r$ with respect to $s$, we have
$$\frac{dr}{ds}P(s) +r P'(s) = Q(s)^r \left (\frac{dr}{ds} \log Q(s) + r \frac{Q'(s)}{Q(s)} \right ).$$
Hence, letting $\frac{dr}{ds}=0$, we have
$$rP'(s) =r Q(s)^r \, \frac{Q'(s)}{Q(s)}. $$
Since $r \not =0$, we have
$$P'(s) =Q(s)^r \, \frac{Q'(s)}{Q(s)}.$$
Therefore, the local maxima $r(s)$ satisfies the following two equations:
\begin{equation*}
{} \begin{cases}
rP(s)=Q(s)^r, & \\
P'(s) =Q(s)^r \, \frac{Q'(s)}{Q(s)}. &\\
\end{cases}
\end{equation*}
From this, we get
$$P'(s) =rP(s) \, \frac{Q'(s)}{Q(s)},$$
namely we get
\begin{equation*}\label{equ-1}
r(s) = \frac{P'(s)Q(s)}{P(s)Q'(s)}.
\end{equation*}
Plugging it in the equation $rP(s)=Q(s)^r$, we get (\ref{equ-0}) above.
\end{proof}
%%%%%%%%$
The domain $\mathcal D(r(s)) = \{ s \in [1, \infty) \, | \, P(s) \ge e\log Q(s) \}$ of the implicit function $r=r(s)$ is 
\begin{equation}\label{dom}
\mathcal D(r(s)) =  [s_1, s_2]  \sqcup [s_3,s_4] \cdots \sqcup [s_{k-2},s_{k-1}] \sqcup [s_k, \infty),
\end{equation}
where $s_2, \cdots, s_k$ are the zeros of the equation $P(s) = e\log Q(s)$ such that $s_2 <s_3 < \cdots <s_k$ and $s_1=1$ or $s_1$ is another zero of $P(s) = e\log Q(s)$, depending on the polynomials $P(s)$ and $Q(s)$.
%%%%%%%%%%%%
\begin{cor} Let $P(s)$ and $Q(s)$ be as above. Let $r=r(s)$ be the implicit function defined by the equation $rP(s) =Q(s)^r$, let $s_1, \cdots,s_k$ be the end points of the closed intervals of the domain $\mathcal D(r(s))$ of the implicit function $r=r(s)$ as in (\ref{dom}), and let $c_1, \cdots, c_m$ be the finite critical points \footnote{It is enough to consider only the critical points whose values are local maxima, but the formula (\ref{max-formula}) is, of course, not affected by considering all the critical points.} of the implicit function $r=r(s)$. 
Then we have
$$\frak{sth}\left (P(s),Q(s);1 \right)= \max \{r(s_1), \cdots, r(s_k), r(c_1), \cdots, r(c_m), \lim_{s \to \infty}r(s) \}.$$
In particular, for a a simply connected elliptic space $X$, we have
\begin{equation}\label{max-formula}
\frak{pp}_{\mathbb R}(X;1) = \max \{r(s_1), \cdots, r(s_k), r(c_1), \cdots, r(c_m), \lim_{s \to \infty}r(s) \}.
\end{equation}
\end{cor}
%%%%%%%%%%%
\begin{rem} We do not know if $\frak{pp}_{\mathbb R}(X;1)$ is attained at an end point $s_j$ of these closed intervals, at a critical point $c_k$, or as the limit $\lim_{s \to \infty}r(s)$. In the case of $X=S^{2n}$, it is not obtained as $\lim_{s \to \infty}r(s)$, as shown in the following section.
\end{rem}
%%%%%%%%%%%
As computed in the following section, we have that $\frak{pp}_{\mathbb R}(S^{2n+1};1)=1$. Except the case of $S^{2n+1}$, we make the following conjecture:
%%%%%%%%%
\begin{con} Let $X$ be a simply connected elliptic space. If $X$ is not homotopic to $S^{2n+1}$, then the real stabilization threshold $\frak{pp}_{\mathbb R}(X;1)$ is a transcendental number. In particular, if $\frak{pp}_{\mathbb R}(X;1)$ is attained at a critical point $s$ of the implicit function $r(s)$, then the following number is a transcendental number:
$$
\frac{P'(s)Q(s)}{P(s)Q'(s)}
$$
where $P(s)=P^{\pi}_X(s)$ and $Q(s)=P_X(s)$. 
\end{con}
%%%%%%%%%%%%%%%
Similarly to Definition \ref{mhp-defn} and Definition \ref{real-pp}, one can define  
the following real stabilization threshold: 
%%%%%%%%%%
\begin{defn}
$$\frak{mhp}_{\mathbb R }(X;[\varepsilon,r], [\varepsilon,r], [\varepsilon,r]):=\inf \{ r \in \mathbb R \vert
qMH^{\pi}_X(t,u,v)<MH_X(t,u,v)^q, \forall q \ge r, \forall (t,u,v) \in \mathscr C_{\varepsilon, r} \}.$$
\end{defn}
%%%%%%%%%%
A natural question is if the real stabilization threshold $\frak{mhp}_{\mathbb R }(X;[\varepsilon,r], [\varepsilon,r], [\varepsilon,r])$ stabilizes 
for $r \rightarrow \infty$, namely if the following real stabilization threshold 
$$\frak{mhp}_{\mathbb R }(X;\varepsilon,\varepsilon,\varepsilon):=\inf\{ r\in \mathbb R \vert
qMH^{\pi}_X(t,u,v)<MH_X(t,u,v)^q, \forall q \ge r, \forall (t,u,v) \in \mathscr C_{\varepsilon, \infty}\}$$
 is well-defined. We have the following proposition, showing that
this is the case under certain condition. 
%%%%%%%%%%%%
\begin{pro} \label{pro-real-mhp} Let $P(t,u,v)$ and $Q(t,u,v)$ be two polynomials with positive
integer coefficients such that there exist polynomials $\phi(t,u,v),\widetilde
P(s)$ and $\widetilde Q(s)$ such that 
\begin{itemize}
\item $P(t,u,v) \le \tilde P \left (\phi(t,u,v) \right )$,
\item $\tilde Q \left (\phi(t,u,v) \right ) \le Q(t,u,v)$,
\item $\exists \, \, \eta$ such that $\phi(t,u,v) \ge \eta>0, \forall (t,u,v) \in \mathscr C_{\varepsilon, \infty}= [\varepsilon,\infty)^3$.
\end{itemize} 
Then there exist 
$r_0 \in \RR$ such that 
for $\forall (t,u,v) \in [\varepsilon, \infty)^3$ and 
$\forall r \ge r_0$  one has 
$$rP(t,u,v) <Q(t,u,v)^r.$$
In particular, if there exist such polynomials for the homotopical and homological mixed Hodge polynomials $P(t,u,v):=MH^{\pi}_X(t,u,v)$ and $Q(t,u,v):=MH_X(t,u,v)$, then the real stabilization threshold 
$\frak{mhp}_{\mathbb R }(X;\varepsilon,\varepsilon,\varepsilon)$
is well defined. 
\end{pro}
%%%%%%%%%%%%%
\begin{proof} Indeed, let 
$r_0=\frak{sth}_{\mathbb R }\left (\widetilde P(s), \widetilde Q(s); \eta \right )$ be the real stabilization threshold for the
polynomials $\tilde P(s)$ and $\tilde Q(s)$, i.e.,
$\frak{sth}_{\mathbb R }\left (\widetilde P(s), \widetilde Q(s); \eta \right)= \inf\{r \vert \, q\widetilde P(s) < \widetilde Q(s)^q, \forall q \ge r, \forall s \ge \eta \}.$
Then for $r \ge r_0$ one has 
$r\widetilde P(s) < \widetilde Q(s)^r$ for $\forall s \ge \eta$ and hence for $\forall (t,u,v)
\in [\epsilon,\infty)^3$ one has 
$$rP(t,u,v) \le r\tilde P(\phi(t,u,v)) <\tilde Q(\phi(t,u,v))^r \le Q(t,u,v)^r.$$
Here we note that
$$\frak{sth}_{\mathbb R }\left (P(t,u,v) ,Q(t,u,v) ; \mathscr C_{\varepsilon, \infty} \right) \le \frak{sth}_{\mathbb R }\left (\widetilde P(s), \widetilde Q(s); \eta \right),$$
where we define
$$\frak{sth}_{\mathbb R }\left (P(t,u,v) ,Q(t,u,v) ; \mathscr C_{\varepsilon, \infty} \right):= \inf\{r \vert \, qP(t,u,v) < Q(t,u,v)^q, \forall q \ge r, \forall (t,u,v) \in \mathscr C_{\varepsilon, \infty} \}.$$
\end{proof}
%%%%%%%%%%%%
\begin{ex}\label{ex-1}
Let us set (see (\ref{cpn-pi}) and (\ref{cpn-homo})):
$$P(t,u,v):=MH_{\CC\PP^n}^{\pi}(t,u,v) = t^2uv + t^{2n+1}(uv)^{n+1},$$
$$Q(t,u,v):=MH_{\CC\PP^n}(t,u,v) = 1+t^2uv + t^4(uv)^2 + \cdots + t^{2n}(uv)^n.$$
We let 
$$\phi(t,u,v)=t^2(uv), \, \widetilde P(s)=s+s^{n+1},\, \widetilde Q(s)=1+s+s^2+...+s^n.$$
Then one has 
\begin{itemize}
\item $MH_{\CC\PP^n}^{\pi}(t,u,v) <\widetilde P(\phi(t,u,v))$ for $\forall t \ge 1$,
\item $\widetilde Q(\phi(t,u,v)) =MH_{\CC\PP^n}(t,u,v)$ for $\forall t \ge 1$, 
\item  Moreover $\phi(t,u,v)=t^2(uv) \ge 1$ for $(t,u,v) \in \mathscr C_{1, \infty}=[1,\infty)^3$. 
\end{itemize}
Hence it follows from Proposition \ref{pro-real-mhp} that the real stabilization threshold $\frak{mhp}_{\mathbb R }(X;1,1,1)$ is well-defined.
\end{ex}
%%%%%%%%%%%%
The following example shows that $\frak{sth}_{\mathbb R }\left (P(t,u,v) ,Q(t,u,v) ; \mathscr C_{\varepsilon, \infty} \right)$ does not necessarily exist.
\begin{ex} We consider the following modification of $P(t,u,v)$ of the above example:
$$P_1(t,u,v)=t^2uv+t^{2n+3}(uv)^{n+1}, \quad \widetilde P_1(s,t)=s+ts^{n+1}.$$
Then we have that $P_1(t,u,v)= \widetilde P_1(\phi(t,u,v), t)$ and $Q(t,u,v)$ be as in Example \ref{ex-1}.
We want to show that $\frak{sth}_{\mathbb R }\left (P(t,u,v) ,Q(t,u,v) ; \mathscr C_{\varepsilon, \infty} \right)$ does not exist.  A key observation is that the two-variable analog $r=r(x,y)$ of the one-variable implicit 
function $r=r(s)$ above, i.e., $r=r(x,y)$ is the largest one of the solutions $r$ of the equation $rx=y^r$, is
an unbounded function in the open region $x>0, y>1$. More precisely, 
for any $y_0>1$ and any $r_0>0$, let us consider any $x_0 > \left (\frac{y_0^{r_0}}{r_0} \right )$ and the line $v= x_0r$ in the plane of coordinates $(r,v)$. Then the largest $r$-coordinate of the two distinct intersections of this line $v= x_0r$ and the exponential curve $v=y_0^r$, i.e., $r(x_0,y_0)$ is bigger than $r_0$, namely, we have that $\exists \, \, r > r_0$ such that $rx_0=y_0^r$.
Therefore we have
$\{ r >0 \, \, \vert \, \, qx < y^q, \forall q >r, \forall x >0, \forall y>1 \} = \emptyset.$
In the above situation, for any $s_0>0$ and any $r_0>0$, let us consider $t_0 >0$ such that
$$\text{$s_0+t_0s_0^{n+1} > \frac{(1+s_0+\cdots +s_0^n)^{r_0}}{r_0}$, \, i.e., \, $t_0> \frac{(1+s_0+\cdots +s_0^n)^{r_0} -r_0s_0}{r_0s_0^{n+1}}$}.$$
Then for $\forall t>t_0$, the largest solution $r$ of the equation $r(s_0+ts_0^{n+1})=(1+s_0+...+s_0^n)^r$
is bigger than $r_0$, namely, $\exists \, \, r > r_0$ such that $r(s_0+ts_0^{n+1})=(1+s_0+...+s_0^n)^r$.
Thus we have
$\{ r >0 \, \, \vert \, \, q(s+ts^{n+1}) < (1+s+\cdots +s^n)^q, \forall q >r, \forall s >0, \forall t>0 \} = \emptyset.$
Therefore  $\frak{sth}_{\mathbb R }\left (P_1(t,u,v) ,Q(t,u,v) ; \mathscr C_{1, \infty} \right)$ does not exist.
\end{ex}
%%%%%%%%%%%%%
%%%%%%%%%%%%%%
%%%%%%%%%%%%%
\section{$\frak{pp}_{\mathbb R}(S^{2n+1};1)$, $\frak{pp}_{\mathbb R}(S^{2n};1)$ and $\frak{pp}_{\mathbb R}(\mathbb CP^n;1)$}

In this section, we try to compute the real stabilization thresholds of $S^{2n+1}$, $S^{2n}$ and $\mathbb CP^n$.

By the same argument as in Example \ref{s-2n} we can show the following
%%%%%%%%%%%%
\begin{lem}\label{lem-1} Let $P(t)$ and $Q(t)$ be polynomials with non-zero integral coefficients and suppose that $P(t)=a_2t^2+ \cdots$ and $Q(t)=1+b_2t^2 + \cdots + b_qt^q$ with $b_q \not =0$ (hence, $Q(t) \ge 2$ for $\forall t \ge 1$). If there exists a positive \emph{integer} $n_0$ such that $n_0P(t) < Q(t)^{n_0}$ for $\forall t \ge 1$, then for all \emph{integer} $n \ge n_0$ we have
$$nP(t) < Q(t)^n\quad (\forall t \ge 1).$$
\end{lem}
%%%%%%%%%%%%
In fact, the part ``for all \emph{integer} $n \ge n_0$" can be replaced by ``for all \emph{real number} $r \ge n_0$", i.e., we can show the following theorem.
%%%%%%%%%%%%
\begin{thm}\label{int-real} Let $P(t)$ and $Q(t)$ be
as in Lemma \ref{lem-1}. If there exists a positive \emph{integer} $n_0$ such that $n_0P(t) < Q(t)^{n_0}$ for $\forall t \ge 1$, then for all \emph{real number} $r \ge n_0$ we have
$$rP(t) < Q(t)^r\quad (\forall t \ge 1).$$
\end{thm}
%%%%%%%%%%%%
To prove this theorem, first we show the following lemma:
%%%%%%%%%%%%
\begin{lem}\label{r-lemma} Suppose that for a positive real number $r_0$
$$r_0P(t)<Q(t)^{r_0} \quad \forall t \ge 1.$$
If $r_0 \ge \frac{1}{\log 2} \, 
(>1)$, where $\log$ is the natural logarithm, then for any real number $r \ge r_0$ the following holds: 
$$rP(t)<Q(t)^r\quad \forall t \ge 1.$$
\end{lem}
%%%%%%%%%
\begin{proof}   
$r_0P(t)<Q(t)^{r_0} (\forall t \ge 1) $ implies 
$P(t)<\frac{1}{r_0}Q(t)^{r_0} (\forall t \ge 1).$
So, for $r \ge r_0$, we have
$rP(t)<\frac{r}{r_0}Q(t)^{r_0} (\forall t \ge 1).$
So, if we can show $\frac{r}{r_0}Q(t)^{r_0} \le Q(t)^r (\forall  t \ge 1)$, then we are done.
In other words, we show that
$$\frac{1}{r_0}Q(t)^{r_0} \le \frac{1}{r}Q(t)^r \quad \forall t \ge 1.$$
To show this, let us consider the function $F(z):= \frac{1}{z}Q(t)^z.$
For a fixed $t (\ge 1)$, we have 
$$F'(z)= - \frac{1}{z^2}Q(t)^z + \frac{1}{z}\left (\log Q(t) \right )Q(t)^z
=\frac{1}{z} \left (\log Q(t) - \frac{1}{z} \right )Q(t)^z.$$
Since $Q(t) \ge 2$ (for all $t \ge 1$), $\log Q(t) \ge \log 2.$
If $z \ge \frac{1}{\log 2}$, then $\frac{1}{z} \le \log 2 \le \log Q(t).$
Hence we have
$$\log Q(t) - \frac{1}{z} \ge 0.$$
Therefore, for $z \ge \frac{1}{\log 2}$, we have
$F'(z) \ge 0  (\forall t \ge 1).$
Hence $z \ge \frac {1}{\log 2}$, $F(z)$ is non-decreasing, thus
for $\forall r \ge r_0$, we have $F(r_0) \ge F(r),$ i.e., 
$$\frac{1}{r_0}Q(t)^{r_0} \le \frac{1}{r}Q(t)^r \quad \forall x \ge 1.$$
\end{proof}
%%%%%%%%%%%
The following corollary follows from the above lemma:
%%%%%%%%%%%%
\begin{cor}\label{coro-2} Let $P(t)$ and $Q(t)$ be  
as in Lemma \ref{lem-1}.  If there exists a positive \emph{integer} $n_0 \ge 2$ such that $n_0P(t) < Q(t)^{n_0}$ for $\forall t \ge 1$, then for all \emph{real number} $r \ge n_0$ we have
$$rP(t) < Q(t)^r\quad (\forall t \ge 1).$$
\end{cor}
%%%%%%%%%%%%
Hence in order to finish the proof of Theorem \ref{int-real}, we need to consider the case when the integer $n_0$ is equal to $1$, i.e., prove the following lemma:
%%%%%%%%%%
\begin{lem} Let $P(t)$ and $Q(t)$ be  
as in Lemma \ref{lem-1}.  If $P(t) < Q(t)$ for $\forall t \ge 1$, then for all \emph{real number} $r \ge 1$ we have
$$rP(t) < Q(t)^r\quad (\forall t \ge 1).$$
\end{lem}
%%%%%%%%%
\begin{proof} First, by the assumption, we note that $Q(t)=1+b_2t^2 + \cdots + b_qt^q$ with $b_q \not =0$. Then we have the following two cases:
%%%%%%%%%%%%%
\begin{enumerate}
\item $Q(1) \ge 3$, thus $Q(t) \ge 3$ for $\forall t \ge 1$, and 
\item $Q(1)=2$, thus $Q(t) \ge 2$ for $\forall t \ge 1.$
\end{enumerate}
%%%%%%%%%%%%
\begin{enumerate}
\item 
In this case, by elementary calculation we see that for any real number $r \ge 1$ 
\begin{equation*}
rQ(t) \le  Q(t)^r \quad \forall t \ge 1
\end{equation*}
Thus $rP(t) <rQ(t) <Q(t)^r.$
 \item 
In this case we have to have $Q(t)=1+t^q$. Since $P(t) <Q(t)=1+t^q$ for $\forall t \ge 1$, we have to have $P(1)<Q(1)=2$, hence $P(1)=1$ is the only possibility, thus $P(t)=t^p$ and $p \le q$ since $t^p=P(t) <Q(t)=1+t^q$ for $\forall t \ge 1$. Then, we see that for any real number $r \ge 1$ we have $rt^p < (1+t^q)^r$, i.e.,
 $rP(t) < Q(t)^r.$
\end{enumerate}
\end{proof}

%%%%%%%%%
\begin{rem}\label{rem-123}
Here we point out that the following modified versions of the above Theorem \ref{int-real} does not hold:
``\emph{If there exists a positive \emph{real number} $r_0$ such that $r_0P(t) < Q(t)^{r_0}$ for $\forall t \ge 1$, then for all \emph{real number} $r \ge r_0$ we have $rP(t) < Q(t)^r\quad (\forall t \ge 1).$"}  Here is a simple example. Let $P(t)=2t, Q(t)=1+t^2$. Then we have $1.6t=0.8P(t) < Q(t)^{0.8} =(1+t^2)^{0.8} (\forall t \ge 1)$ (in fact for $\forall t$). However $P(t) \not <Q(t) (\forall t \ge 1)$, simply because $P(1)=Q(1)=2$.
\end{rem}
%%%%%%%%%%
\begin{pro} $\frak{pp}_{\mathbb R}(S^{2n+1};1)=1$.
\end{pro}
\begin{proof} We have $P_{S^{2n+1}}(t) = 1+t^{2n+1}$ and $P^{\pi}_{S^{2n+1}}(t) = t^{2n+1}$.
Thus $P^{\pi}_{S^{2n+1}}(t) < P_{S^{2n+1}}(t)$. Hence it follows from Theorem \ref{int-real} that for any real number $r \ge 1$ we have 
$$rP^{\pi}_{S^{2n+1}}(t) < P_{S^{2n+1}}(t)^r.$$
 Therefore $\frak{pp}_{\mathbb R}(S^{2n+1};1) \le 1.$
Now we claim that  $\frak{pp}_{\mathbb R}(S^{2n+1};1)=1$, i.e.,
\begin{equation}\label{claim}
\text{ if $r<1$, then $rt^N<(1+t^N)^r \quad (\forall t \ge 1)$ \emph{does not} hold.} 
\end{equation}
Since $rt^{N(1-r)}>1$ for a sufficient large number $t$ and $2^r <1$ for $r<1$, we have that for a sufficiently large number $t$
$$rt^{N(1-r)} >2^r > \left (1 + \frac{1}{t^N}\right )^r=\frac{(1+t^N)^r}{t^{Nr}}.$$
Which implies that for a sufficient large number $t$ we have
$$rt^N = rt^{N(1-r)} \cdot t^{Nr} > (1+t^N)^r.$$
Therefore we get the above claim (\ref{claim}).
\end{proof}
%%%%%%%%%%%%%%%%%%%%
\begin{lem}
$$\frak{pp}_{\mathbb R}(S^{2n};1) \le 2+ \frac{1}{2}.$$
\end{lem}
%%%%%%%%%
\begin{proof}
We show
\begin{equation*}
(1+t^{2n})^{ 2+ \frac{1}{2}} >  \left (2+ \frac{1}{2} \right) (t^{2n} +t^{4n-1}) \quad (\forall t \ge 1).
\end{equation*}
If we can show the following, we are done since $t^{2n} +t^{4n} \ge t^{2n} +t^{4n-1}$.
\begin{equation*}
(1+t^{2n})^{ 2+ \frac{1}{2}} >  \left (2+ \frac{1}{2} \right) (t^{2n} +t^{4n}) \quad (\forall t \ge 1).
\end{equation*}
Let $X:=t^{2n}$ and consider the following function
$$F(X)= (1+X)^{ 2+ \frac{1}{2}} - \left (2+ \frac{1}{2} \right) (X + X^2) \quad (\forall X \ge 1).$$
By elementary calculus we can see that $F(X) >0$ (in fact, for $X \ge 0$).
\end{proof}
%%%%%%%%%%%%
As in Example \ref{s-2n}, $\frak{pp}(S^{2n};1)=3$, clearly $\frak{pp}_{\mathbb R}(S^{2n};1) >2$, thus we have
$$2 < \frak{pp}_{\mathbb R}(S^{2n};1) < 2+ \frac{1}{2}.$$
%%%%%%%%%%%%%
\begin{qu} 
Does there exist a simple estimate of $\frak{pp}_{\mathbb R}(S^{2n};1)$?
\end{qu}
%%%%%%%%%%%%
\begin{rem} Let us consider the implicit function $r=r(s)$ of $r(s^{2n} +s^{4n-1})=(1+s^{2n})^r$. If $s>2$, we do have $(1+s^{2n})^2 > 2(s^{2n} +s^{4n-1})$, hence it follows from Corollary \ref{coro-2} that $(1+s^{2n})^r > r(s^{2n} +s^{4n-1})$ for any real number $r \ge 2$ (for $\forall s>2$). Therefore the implicit function $r=r(s)$ satisfies that $r(s) < 2$ for $\forall s >2$.
Hence $\lim_{s \to \infty}r(s) \le 2$.  Therefore the threshold $\frak{pp}_{\mathbb R}(S^{2n};1)$ is not attained as the limit
$\lim_{s \to \infty}r(s)$. 
\end{rem}
%%%%%%%%%%%
In Lemma \ref{r-lemma} we use the condition that $r_0 \ge \frac{1}{\log 2}$. However, we can relax the condition in the following lemma.
%%%%%%%%%%
\begin{lem}\label{r>1-lemma} Let $P(t)$ and $Q(t)$ be 
as in Lemma \ref{lem-1} and we suppose that $Q(1) \ge 3$, thus $Q(t) \ge 3 $ for $\forall t \ge 1$.
If there exists a positive real number $r_0 >1$ such that $r_0P(t)<Q(t)^{r_0} (\forall t \ge 1)$, then for any real number $r \ge r_0$ we have
$$rP(t)<Q(t)^r\quad \forall t \ge 1.$$
\end{lem}
%%%%%%%%%%
\begin{proof} Since $r_0P(t)<Q(t)^{r_0} (\forall t \ge 1)$, we have $P(t)<\frac{1}{r_0} Q(t)^{r_0} (\forall t \ge 1)$, thus for any $r \ge r_0$ we have $rP(t)<\frac{r}{r_0} Q(t)^{r_0} (\forall  t \ge 1)$. So, it suffices to show 
$$\text{$\frac{r}{r_0} Q(t)^{r_0} \le Q(t)^r (\forall t \ge 1)$, i.e., $\frac{1}{r_0} Q(t)^{r_0} \le \frac{1}{r}Q(t)^r (\forall t \ge 1)$}.$$
In other words it suffices to show that $F(r):=\frac{1}{r}Q(t)^r$ is an increasing function for $r>1$ for $\forall t \ge 1$, which is easy to see.
\end{proof}
%%%%%%%%%%%
\begin{pro}\label{1/3n-3/2n} For $\forall n \ge 3$ we have
$$1 + \frac{1}{3n} \le \frak{pp}_{\mathbb R}(\mathbb CP^n;1)  \le 1 + \frac{3}{2n}.$$
\end{pro}
%%%%%%%%
\begin{proof} Let $n \ge 3$. First we show
\begin{equation}\label{3/2n}
(1+t^2+t^4+ \cdots +t^{2n})^{1+\frac{3}{2n}} > \left (1+\frac{3}{2n} \right )(t^2 + t^{2n+1}) \quad (\forall t \ge 1).
\end{equation}
Indeed, we have
\begin{align*}
& (1+t^2+\cdots +t^{2n-2}+ t^{2n})^{1+\frac{3}{2n}} - \left (1+\frac{3}{2n} \right )(t^2 + t^{2n+1}) \\
& \ge  (1+t^2+t^{2n-2} +t^{2n})(1+t^2+t^{2n-2} +t^{2n})^{\frac{3}{2n}} - 2(t^2 + t^{2n+1}) \\
& > (1+t^2+t^{2n-2} +t^{2n})(t^{2n})^{\frac{3}{2n}} - 2(t^2 + t^{2n+1}) \\
& = (1+t^2)t^3 -2t^2 + (1+t^2)t^{2n+1} - 2t^{2n+1}\\
& \ge 2t^3-2t^2\, \, (\forall t \ge 1)\\
& \ge 0.
\end{align*}
Hence it follows from Lemma \ref{r>1-lemma} that $\frak{pp}_{\mathbb R}(\mathbb CP^n;1)  \le 1 + \frac{3}{2n}.$
Next we show that there exists some $s\ge 1$ such that
\begin{equation}\label{1/3n}
\left (1+\frac{1}{3n} \right )(s^2 +s^{2n+1}) > (1+s^2+ \cdots +s^{2n})^{1+\frac{1}{3n}}.
\end{equation}
This implies that $1 + \frac{1}{3n} \le \frak{pp}_{\mathbb R}(\mathbb CP^n;1)$.
Let $t >2$. Then we first observe that
\begin{equation*}
1+t^2+t^4 + \cdots +t^{2n} <2t^{2n},
\end{equation*}
which follows from
\begin{equation*}
1+t^2+t^4 + \cdots +t^{2n} = \frac{t^{2n+2}-1}{t^2-1} < \frac{t^{2n+2}}{t^2-1} = \left (\frac{t^2} {t^2-1} \right )t^{2n} < 2t^{2n} \quad \text{(since $t>2$)}.
\end{equation*}
Now we have
\begin{align*}
& \left (1+\frac{1}{3n} \right )(t^2 +t^{2n+1}) - (1+t^2+ \cdots +t^{2n})^{1+\frac{1}{3n}}\\
& > \left (1+\frac{1}{3n} \right )(t^2 +t^{2n+1}) - (2t^{2n})^{1+\frac{1}{3n}}\\
& > t^{2n+1} - 2^{1+\frac{1}{3n}}t^{2n+ \frac{2}{3}}\\
& = t^{2n+ \frac{2}{3}} \left (t^{\frac{1}{3}} - 2^{1+\frac{1}{3n}} \right )\\
& > 0 \quad \text{for $\forall t> \left (2^{1+\frac{1}{3n}} \right )^3$}.
\end{align*}
Hence if $s>\left (2^{1+\frac{1}{3n}} \right )^3$, then we have the above inequality (\ref{1/3n}).
\end{proof}
%%%%%%%%%%%%%%%%%%%%%%%%%%%%

{\bf Acknowledgements:} We would like to thank W. Zudilin, R. Tijdeman, L. Maxim and the referee
 for suggestions and comments. S.Y. is supported by JSPS KAKENHI Grant Number JP19K03468. \\

%%%%%%%%%%%%

\end{document}